\documentclass[a4paper,10pt]{amsart}
\usepackage{amsmath,amsthm,amssymb,latexsym,enumerate,xcolor,hyperref}
\usepackage{graphicx}
\usepackage{eucal}

\numberwithin{equation}{section}

\newtheorem{thm}{Theorem}[section]

\newtheorem{lem}[thm]{Lemma}

\newtheorem{rem}[thm]{Remark}

\newcommand{\RR}{\mathbb{R}}

\newcommand{\Erfc}{\mathop{\mathrm{Erfc}}}

\newcommand{\link}{\mathop{\circ\kern-.35em -}}

\newcommand{\ol}{\overline}
\newcommand{\pa}{\partial}

\newcommand{\lan}{\langle}
\newcommand{\ran}{\rangle}
\newcommand{\tr}{\mathop{\mathrm{tr}}}

\newcommand{\na}{\nabla}

\newcommand{\be}{\beta}
\newcommand{\ga}{\gamma}  
\newcommand{\Ga}{\Gamma}
\newcommand{\de}{\delta}
\newcommand{\De}{\Delta}
\newcommand{\ve}{\varepsilon}
\newcommand{\fhi}{\varphi} 
\newcommand{\la}{\lambda}
\newcommand{\La}{\Lambda}    
\newcommand{\ka}{\kappa}
\newcommand{\si}{\sigma}

\newcommand{\te}{\theta}
\newcommand{\zi}{\zeta}
\newcommand{\om}{\omega}
\newcommand{\Om}{\Omega}

\newcommand{\cA}{\mathcal{A}}

\newcommand{\cL}{\mathcal L}
\newcommand{\cM}{\mathcal M}

\newcommand{\cS}{\mathcal S}

\title[Asymptotics 
for Pucci operators]{Small diffusion and short-time asymptotics \\
for Pucci operators}

\author{Diego Berti}
\address{CSIC --- Consejo Superior de Investigaciones Cient\'ificas, ICMAT --- Instituto de Ciencias Matem\'aticas, Calle Nicol\'as Cabrera 13-15,
Campus de Cantoblanco (UAM)
28049 Madrid, Spain}
    \email{diego.berti@icmat.es}
\urladdr{https://www.icmat.es/diego.berti}

\author[R. Magnanini]{Rolando Magnanini}
\address{Dipartimento di Matematica ``U. Dini'', Universit\`a di Firenze, viale Morgagni 67/A, 50134 Firenze, Italy}
\email{magnanini@unifi.it}
\urladdr{http://web.math.unifi.it/users/magnanin}

 \date{}
 
\begin{document}
 
 \begin{abstract}
This paper presents asymptotic formulas in the case of the following two problems for the {\it Pucci\rq{}s extremal operators} $\cM^\pm$. It is considered the solution $u^\ve(x)$ of $-\ve^2 \cM^\pm\left(\na ^2 u^\ve\right)+u^\ve=0$ in $\Om$ such that $u^\ve=1$ on $\Ga$. Here, $\Om\subset \RR^N$ is a domain (not necessarily bounded) and $\Ga$ is its boundary. It is also considered $v(x,t)$ the solution of $v_t - \cM^\pm\left(\na^2 v\right)=0$ in $\Om\times (0,\infty)$, $v=1$ on $\Ga\times(0,\infty)$ and $v=0$ on $\Om\times \{0\}$.  In the spirit of their previous works \cite{BM-AA, BM-JMPA}, the authors establish the profiles as $\ve$ or $t\to 0^+$ of the values of $u^\ve(x)$ and $v(x,t)$ as well as of those of their $q$-means on balls touching $\Ga$. The results represent a further step in the extensions of those obtained by Varadhan and by Magnanini-Sakaguchi in the linear regime.
 \end{abstract}
 
 
  \maketitle
 
 \section{Introduction}

Varadhan's formulas are now more than fifty years old. Their original motivation has to do with the asymptotic behavior of probabilities. The two important reference situations concern an elliptic boundary value problem,
$$
\ve^2 \cL[u^\ve]=u^\ve \ \mbox{ in } \ \Om, \quad u^\ve=1 \ \mbox{ on } \ \Ga,
$$
and a parabolic initial-boundary value problem,  
$$
v_t- \cL[v]=0 \ \mbox{ in } \ \Om\times (0,\infty), \quad v=0 \ \mbox{ on } \ \Om\times \{ 0\}, \quad v=1 \ \mbox{ on } \ \Ga\times (0,\infty).
$$
Here, $\Om$ is a domain in $\RR^N$, $N\ge 2$, not necessarily bounded and with sufficiently regular boundary $\Ga$, and $\cL$ is an elliptic operator. In both cases, the maximum principle gives that the values of $u^\ve$ and $v$ in $\Om$ belong to the interval $(0,1)$, thus giving grounds for a probabilistic interpretation for them.
\par
In his seminal paper \cite{Va}, Varadhan considers a uniformly elliptic linear operator $\cL$ with H\"older continuous coefficients and proves the two formulas:
$$
\lim_{\ve\to 0^+} \ve \log u^\ve(x)=-d_\cL(x,\Ga), \ x\in\ol{\Om},
$$
and
$$
\lim_{t\to 0^+} 4t \log v(x,t)=-d_\cL(x,\Ga)^2, \ x\in\ol{\Om}.
$$
Here, we denote by $d_\cL(x,\Ga)$ the shortest distance, induced by a Riemannian metric derived from the 
coefficients of $\cL$, to points on the boundary $\Ga$
from $x$. 
These two asymptotic formulas express in a precise manner that the two families of solutions present a boundary layer with exponential profile when $\ve$ and $t\to 0^+$.
\par
Varadhan's formulas have been generalized by the second author and some of his co-authors to various nonlinear regimes. The formulas have to be modified by possibly replacing the logarithm by a more suitable profile that depends on the nonlinearity present in the operator $\cL$. In \cite{MS-PRSE} and \cite{Sa}, the case of the $p$-Laplace operator is considered; \cite{MS-AIHP} deals with general fast diffusion non-degenerate operators; \cite{BM-JMPA} and \cite{BM-AA} have to do with the \textit{game theoretic $p$-laplacian}.

Besides their importance in probability theory, Varadhan-type formulas find applications to the study of \textit{isothermic or time invariant level surfaces}. These are surfaces contained in $\Om$ that are spatial level surfaces for $v(\cdot,t)$ for each given time $t$ (invariant surfaces for the family of solutions $u^\ve$
can also be defined similarly). 
One remarkable property of this kind of invariant surfaces is that they are \textit{parallel} (in the relevant metric) to the boundary $\Ga$, that is the points on them have the same distance to $\Ga$. This fact clearly descends from Varadhan-type formulas, since the left-hand sides (and hence the right-hand sides) of them by definition do not depend on the position of the point on the relevant invariant surface. The fact that $d_{\cL}(x,\Ga)$ and the solutions $u^\ve$ or $v(\cdot,t)$ share the same level surface has been used in \cite{CMS-JEMS}, \cite{MS-AIHP}, \cite{MS-MMAS}, together with the method of moving planes,  to show that compact invariant surfaces are spheres.
\par
Another significant consequence of Varadhan-type formulas, that also entails invariant surfaces (and other stationary objects, such as stationary hot spots), are formulas that even more associate the behavior of solutions for small values of the parameters $\ve$ and $t$ to the geometry of the domain. The first instances of these formulas were given in \cite{MS-AM} and \cite{MS-AN} by the second author of this note and S. Sakaguchi. In this introduction, we only present them for the case of the functions $u^\ve$, when $\cL$ is the Laplace operator. In \cite{MS-AM}, it is proved:
\begin{equation}
\label{MS-formula}
\lim_{\ve\to 0^+}\left(\frac{R}{\ve}\right)^\frac{N+1}{2}\int_{B_R(x)} u^\ve(y)\,dy=c_N\,\Pi_0(z_x)^{-\frac12}.
\end{equation}
Here, $c_N$ is a numerical constant, $B_R(x)$ is a ball contained in $\Om$ and such that $\ol{B_R(x)}\cap\left(\RR^N\setminus\Om\right)=\{z_x\}$, and $\Pi_m:\Ga\to\RR$ is defined by
$$
\Pi_m=\begin{cases}
\prod\limits_{j=m+1}^{N-1} (1-R\,\ka_j) \ &\mbox{ for } \ m=0, 1, \dots, N-2, \\
1 \ &\mbox{ for } \ m=N-1,
\end{cases}
$$
where $\ka_j$, $j=1,\dots, N-1$, are the principal curvatures of $\Ga$ (at points in $\Ga$). Of course, $\Ga$ is assumed to be at least of class $C^2$; $\Ga$ may also be non-compact. 
\par
When the set $\ol{B_R(x)}\cap\left(\RR^N\setminus\Om\right)$ is larger, a more accurate formula in \cite{MS-AN} gives:
\begin{equation}
\label{MS-formula-2}
\lim_{\ve\to 0^+}  \left(\frac{R}{\ve}\right)^\frac{N-1-m}{2}
\int_{\pa B_R(x)} \fhi(y)\, u^\ve(y)\,dS_y=
c_{N,m} \int_M \fhi(y)\,
\Pi_m(y)^{-\frac12} d M_y.
\end{equation}
Here, $c_{N,m}$ is a numerical constant, $M$ is a connected component of $\pa B_R(x)\cap\Ga$ and is an $m$-dimensional submanifold of $\pa B_R(x)$ $(0\le m\le N-1)$ with possibly non-empty boundary,   
$\fhi$ is any continuous function whose support does not intersect the closure of $\pa B_R(x)\cap(\Ga\setminus M)$, and
$dM_y$ is the volume element on the submanifold $M$. When $m=0$, $M$ and $dM_y$ are regarded as a point and the Dirac measure at that point.
\par
Formulas like \eqref{MS-formula} and \eqref{MS-formula-2} and their parabolic counterparts have been used to show that, under sufficient assumptions, non-compact invariant surfaces have planar or cylindrical symmetry (see \cite{MS-IUMJ}, \cite{MS-JDE}, \cite{MPrS}, \cite{MPeS}), or that certain convex polygons are invariant under the action of some specific groups of rotations if a stationary hot spot is present (see \cite{MS-AN}, \cite{MS-JAM}). 
\par
The purpose of this paper is to investigate on similar asymptotic formulas for the so-called {\it Pucci's extremal operators} (see \cite{Pu}). These are fully nonlinear operators that can be defined for every $X$ in the space $\cS^N$ of $N\times N$ symmetric matrices by the formulas
\begin{equation}
\label{def-Pucci}
\cM^-(X) = \La \sum_{\la_i<0} \la_i + \la \sum_{\la_i>0} \la_i \quad \mbox{ and } \quad
\cM^+(X) = \la \sum_{\la_i<0} \la_i+ \La \sum_{\la_i >0} \la_i,
\end{equation}
being $\la_i=\la_i(X)$, $i=1, \dots, N$ the eigenvalues of $X$.
Here, $\la$ and $\La$ are given numbers such that $0<\la\le\La$. 
We shall thus consider the respective solutions $u^\ve_\pm$ and $v^\pm$ of the problems
\begin{equation}
\label{elliptic-Pucci}
-\ve^2 \cM^\pm(\na^2 u)+u=0  \ \mbox{ in } \ \Om, \quad u=1 \ \mbox{ on } \ \Ga,
\end{equation}
and
\begin{equation}
\label{parabolic-Pucci}
v_t- \cM^\pm(\na^2 v)=0 \ \mbox{ in } \ \Om\times (0,\infty), \ \ v=0 \ \mbox{ on } \ \Om\times \{ 0\}, \ \ v=1 \ \mbox{ on } \ \Ga\times (0,\infty),
\end{equation}
and study their behavior as $\ve$ or $t\to 0^+$. Here, we mean that the differential equations in \eqref{elliptic-Pucci} and \eqref{parabolic-Pucci} are satisfied according to the theory of viscosity solutions (see \cite{CIL}). Also, we specify that if $\Om$ is unbounded, we consider only the bounded solution of \eqref{elliptic-Pucci} or \eqref{parabolic-Pucci}.
\par
Pucci's extremal operators emerge in the study of stochastic control in the case in which
the diffusion coefficient is a control variable (\cite{BL-82}, \cite{Li-83a}, \cite{Li-83b}, \cite{Li-8182}). They have also been used to provide a natural definition of uniform ellipticity for fully nonlinear operators in the theory of viscosity solutions. In fact, a fully nonlinear operator $F:\Om\times\RR\times\RR^N\times\cS^N\to\RR$ is said to be uniformly elliptic if 
$$
\cM^-(X-Y)\le F(x,s,\xi, Y)-F(x,s,\xi, X)\le \cM^+(X-Y),
$$
for any $(x,s,\xi)\in \Om\times\RR\times\RR^N$ and $X, Y\in\cS^N$ (see \cite{CIL}, \cite{Ko}).
\par
Despite their full nonlinearity, $\cM^\pm$ share some useful features with the already mentioned game-theoretic $p$-laplacian $\De_p^G$, that is instead quasi-linear, since its action on a given function $u$ can be formally defined by 
$$
p\,\De_p^G u=\De u+(p-2)\,\frac{\lan\na^2 u\na u, \na u\ran}{|\na u|^2}
$$
(notice that for $p=2$, $\De_p^G$ coincides with $\De/2$). Indeed, besides being uniformly elliptic, $\cM^\pm$ and $\De_p^G$ are both positively $1$-homogeneous and rotation invariant but, more importantly, if $\Om$ is either a half-space, a ball, or the exterior of a ball, the solutions of  \eqref{elliptic-Pucci} and \eqref{parabolic-Pucci} can be retrieved by some relevant changes from those obtained in \cite{BM-AA} and \cite{BM-JMPA}  with $\De_p^G$ in place of $\cM^\pm$. 
\par
With these remarks in mind, we now present the main results in this paper. In what follows, given a (positive strictly increasing) modulus of continuity $\om$, we say that an open set $\Om$ is of class $C^{0,\om}$ if its boundary $\Ga$ is locally the graph of a continuous function with modulus $\om$.
Associated with $\om$, we will consider the function $\psi_\om:[0,\infty)\to [0,\infty)$ defined at $\si\geq 0$ as the distance of the point $(0,\si)$ to the graph of $\om$ (for details see \eqref{psi function}).
Also, we shall denote for short by $d_\Ga(x)$ the shortest Euclidean distance from $x$ to points on the boundary $\Ga$. 

\begin{thm}[Small diffusion asymptotics]
\label{th:elliptic-Pucci}
Let $\Om$ be a bounded open set and let $u^\ve_\pm$ be the respective solutions of \eqref{elliptic-Pucci}.
The following claims hold true.
\par
\begin{enumerate}[(i)]
\item
If $\Ga = \pa \left(\RR^N\setminus \ol\Om\right)$, then we have that
$$
\lim_{\ve\to 0^+} \ve \log u^\ve_-(x) = -\frac{d_\Ga(x)}{\sqrt{\la}}, \quad \lim_{\ve\to 0^+} \ve \log u^\ve_+(x) = -\frac{d_\Ga(x)}{\sqrt{\La}},
$$
for any $x\in\ol\Om$. 

\item
If $\Om$ is of class $C^{0,\om}$, then as $\ve\to 0^+$
it holds that
\begin{equation*}
\ve \log u^\ve_- + \frac{d_\Ga}{\sqrt{\la}} =
\begin{cases}
 O\left(\ve \log |\log\psi_\om(\ve)|\right) \ &\mbox{ if } \ N=2 \ \mbox{ and } \ \la=\La,\\
 O\left( \ve \log\psi_\om(\ve)\right) \ &\mbox{ if } \ N\neq 2 \ \mbox{ or } \ \la\neq \La,
\end{cases}
\end{equation*}
\begin{equation*}
\ve \log u^\ve_+ + \frac{d_\Ga}{\sqrt{\La}} =
\begin{cases}
O\left(\ve \log \ve\right) \ &\mbox{ if } \ \La > \la(N-1),\\
O\left(\ve \log |\log\psi_\om(\ve)|\right) \ &\mbox{ if } \ \La = \la(N-1),\\
O\left( \ve \log\psi_\om(\ve)\right) \ &\mbox{ if } \ \La < \la(N-1),\\
\end{cases} 
\end{equation*}
uniformly on every compact subset of $\ol\Om$.
\end{enumerate}
\end{thm}

In the parabolic regime, we obtain a somewhat weaker result. 

\begin{thm}[Short-time asymptotics]
\label{th:parabolic-Pucci}
Let $\Om$ be a bounded open set and let $v^\pm$ be the respective solutions of \eqref{parabolic-Pucci}.
The following claims hold true.
\begin{enumerate}[(i)]
\item
If $\Ga = \pa\left(\RR^N\setminus \ol{\Om} \right)$, then we have that
\begin{equation}
\label{parabolic pucci asymp}
\lim_{t\to 0^+}
 4 t\log v^-(x,t) = -\frac{d_\Ga(x)^2}{\la}, \quad  \lim_{t\to 0^+} 4 t\log v^+(x,t) = -\frac{d_\Ga(x)^2}{\La},
\end{equation}
for every $x\in\ol \Om$.
\item
If $\Om$ is of class $C^{0,\om}$, then as $t\to 0^+$ we have that
\begin{eqnarray}
\label{uniform ppucci asymp}
&&4t\log v^-(x,t) + \frac{d_\Ga(x)^2}{\la}= O\left(t \log \psi_\om(t)\right), \\
\label{uniform ppucci plus}
&&4t\log v^+(x,t) + \frac{d_\Ga(x)^2}{\La} = O\left(t \log \psi_\om(t)\right),
\end{eqnarray}
uniformly on every compact subset of $\ol\Om$.
\end{enumerate}
\end{thm}

Once Theorems \ref{th:elliptic-Pucci} and \ref{th:parabolic-Pucci} are settled, we can easily derive formulas similar to \eqref{MS-formula} for Pucci operators. We shall present them in Section \ref{sec:q-mean-formulas}. In Section \ref{sec:asymptotics-with-symmetry}, we shall derive our asymptotic formulas in some spherically symmetric domains, in which solutions can be explicitly computed. These formulas will then be used to construct barriers for the problems in general domains in Section \ref{sec:barriers}. The proofs of Theorems \ref{th:elliptic-Pucci} and \ref{th:parabolic-Pucci} will be carried out in Subsections \ref{ssec:elliptic varadhan formula} and \ref{ssec:parabolic varadhan formula}, respectively.

This paper is dedicated to our friend and  colleague Sergio Vessella on the occasion of his $65^{th}$ birthday. Carlo Pucci, the inventor of the eponymous operators, was the advisor of the second author and Sergio's mentor. We thought that this paper could be an ideal gift to Sergio.

\section{Small diffusion asymptotics in symmetric domains}
\label{sec:asymptotics-with-symmetry}

\subsection{Preliminaries on Pucci operators}

As already mentioned, the operators $\cM^\pm$ defined in \eqref{def-Pucci} are fully nonlinear, in the sense that they are nonlinear in the variable $X$. Also, $-\cM^\pm$ are uniformly elliptic by definition, and hence (degenerate) elliptic for the theory of viscosity solutions, as shown in \cite{CIL} or by direct inspection. The positive homogeneity is evident. Another equivalent definition of $\cM^\pm$ can be given by introducing the set $\cA_{\la, \La}$ of all matrices $A\in\cS^N$ such that $\la\, I\le A\le\La\, I$, that means that $\la\,|\zi|^2\le\lan A \zi, \zi\ran\le\La\,|\zi|^2$ for every $\zi\in\RR^N$. In fact, it holds that
$$
\cM^-(X) = \inf_{A\in\cA_{\la,\La}} \tr \left(A X\right), \quad
\cM^+(X) = \sup_{A\in\cA_{\la,\La}} \tr \left(A X\right).
$$
From this definition, we easily infer in particular that
$$
\cM^- \left(\na^2 u\right)\le\De_p^G u\le \cM^+ \left(\na^2 u\right), 
$$
for $\la=\min\big\{\frac1{p}, \frac{p-1}{p}\big\}$ and $\La=\max\big\{\frac1{p}, \frac{p-1}{p}\big\}$.

For an extensive overview of the main properties of the operators $\cM^\pm$ defined in \eqref{def-Pucci}, see \cite{CC}. 

 \subsection{Pucci operators on radial functions}
 Assume that $u$ is a spherically symmetric function, namely $u(x)=u(r)$, where $r=|x|$.  We can explicitly calculate $\cM^\pm\left(\na^2 u\right)$ in terms of the radial derivatives $u_r$ and $u_{rr}$:
 \begin{align}
  \label{radial pucci minus}
  \cM^-\left(\na^2 u\right) = \be\left(u_{rr}\right) + \frac{N-1}{r} \be\left(u_r\right),
  \\
 \label{radial pucci plus}
  \cM^+\left(\na^2 u\right) = \ga\left(u_{rr}\right) + \frac{N-1}{r} \ga\left(u_r\right),
 \end{align}
 where $\be(\si)=\min(\la \si, \La \si)$ and $\ga(\si)=\max(\la \si, \La \si)=-\be(-\si)$ for $\si\in \RR$.
We note that the functions $\ga$ and $\be$ are just linear in the case their arguments do not change sign.

\subsection{Radial solutions of problem \eqref{elliptic-Pucci}}
 
%

We summarize \cite[Lemmas 2.3, 2.4]{ThesisBerti} in the following technical lemma.

\begin{lem}[Modified Bessel functions]
\label{lem:bessel}
Let two numbers $a>0$ and $b>-1$ be given and let $f, g:[0, \infty) \to (0,\infty)$ be the functions defined by
$$
g(\si)= \int_{0}^{\pi} e^{a \si \cos\te} \left(\sin\te\right)^b d\te, \quad
f(\si) = \int_{0}^{\infty} e^{-a \si \cosh\te} \left(\sinh\te\right)^b d\te,
$$
for any $\si \ge 0$.
\par
Then, $f$ and $g$ are both solutions in $(0,\infty)$ of the equation
$$
-h'' - \frac{b+1}{\si}\, h' + a^2 h=0,
$$
and are such that $g', g''\ge 0$ and $f'\le 0$, $f''\ge 0$. 
\par
Moreover, we have that
\begin{equation}
\label{eq:asymptotics-to-infinity-I}
g(\si) = 2^{\frac{b -1}{2}}\Ga\left(\frac{b+1}{2}\right)  (a\si)^{-\frac{b+1}{2}}  e^{a\si} \left\{1+ O(1/\si)\right\} ,
\end{equation}
\begin{equation}
\label{eq:asymptotics-to-infinity}
f(\si) =
2^\frac{b-1}{2} \Ga\left(\frac{b+1}{2}\right)(a\si)^{-\frac{b+1}{2}} e^{-a\si} \bigl\{ 1+O(1/\si)\bigr\},
\end{equation}
as $\si \to \infty$, and
\begin{equation}
\label{eq:asymptotics-to-zero}
f(\si)=\begin{cases}
\left(a\si\right)^{-b}\, \Ga\left(b\right)\bigl\{ 1+o(1)\bigr\}  \ &\mbox{ if } \ b>0, 
\vspace{6pt}
\\
-\log(a\si)+O(1) \ &\mbox{ if } \ b=0, 
\vspace{2pt}
\\
\frac{\sqrt{\pi}}{2\,\sin(b\pi/2)}\,\frac{\Ga\left(\frac{b+1}{2}\right)}{\Ga\left(\frac{b}{2}+1\right)}+o(1)  \ &\mbox{ if } \ -1<b<0,
\end{cases}
\end{equation}
as $\si\to 0^+$.
\end{lem}

We now derive the solutions and their relevant asymptotics for the case of a ball.

\begin{lem}[Solutions in the ball]
\label{lem:solution ball}
Let $\Om =B_R(0)$. Then the solutions of \eqref{elliptic-Pucci} are given by the functions $u_\pm^\ve$ defined by 
\begin{equation}
\label{radial solution}
u^\ve_-(x)= \frac{\int_{0}^{\pi} e^{ \frac{|x|}{\ve} \frac{\cos \te}{\sqrt{\la}}} \left(\sin\te\right)^{N-2} d\te}
{\int_{0}^{\pi} e^{ \frac{R}{\ve} \frac{\cos \te}{\sqrt{\la}}} \left(\sin\te\right)^{N-2} d\te}\, , \quad
u^\ve_+(x)= \frac{\int_{0}^{\pi} e^{ \frac{|x|}{\ve} \frac{\cos \te}{\sqrt{\La}}} \left(\sin\te\right)^{N-2} d\te}
{\int_{0}^{\pi} e^{ \frac{R}{\ve} \frac{\cos \te}{\sqrt{\La}}} \left(\sin\te\right)^{N-2} d\te}\, ,
\end{equation}
for any $x\in\ol\Om$.
Moreover, we have that
\begin{equation*}
\ve \log u^\ve_- + \frac{d_\Ga}{\sqrt{\la}} = O\left( \ve \log \ve\right), \quad \ve \log u^\ve_+ + \frac{d_\Ga}{\sqrt{\La}} = O\left( \ve \log \ve\right),
\end{equation*}
uniformly on $\ol\Om$ as $\ve\to 0^+$.
\end{lem}

\begin{proof}
By uniqueness, it is sufficient to check that $u^\ve_\pm$ satisfies \eqref{elliptic-Pucci}. We just verify the case of $u^\ve_-$. 
\par
We write $u^\ve_-(x)= u(r)$ with $r=|x|$. Since $u(x) = g(r)/g(R)$, where $g$ is given in Lemma \ref{lem:bessel} with $a= 1/(\sqrt{\la}\,\ve)$ and $b= N-2$, we have that
\begin{equation*}
u_{rr} + (N-1)\frac{u_r}{r} = \frac{u}{\ve^2 \la} \ \mbox{ in } \ (0,R).
\end{equation*}
Since both $u_r$ and $u_{rr}$ are positive in $(0,R)$ (see Lemma \ref{lem:bessel}), then  we have that
$$
\cM^-\left(\na ^2 u^{\ve}_-(x)\right)= \la\, \left\{ u_{rr}(r) + (N-1)\frac{u_r(r)}{r}\right\}=\ve^{-2} u(r)=\ve^{-2} u^\ve_-(x),
$$
and hence $u^\ve_-$ satisfies the differential equation in \eqref{elliptic-Pucci} corresponding to the negative superscript. A direct inspection also gives that $u^\ve_-=u(R)=1$ on $\Ga$.
\par 
Now, we observe that 
\begin{equation*}
u^\ve_-(x)\,e^{a (R-|x|)} = \\
\frac{ \int_{0}^{\pi} e ^{-a |x| (1-\cos\te)} \left(\sin\te\right)^{N-2}\,d\te }
{ \int_{0}^{\pi} e ^{-a R (1-\cos\te)} \left(\sin\te\right)^{N-2}\,d\te }
\end{equation*}
is monotonic with respect to $|x|$. Thus,
\begin{equation*}
 1\leq 
u^\ve_-(x)\,e^{a (R-|x|)} \le \frac{\int_{0}^{\pi} \left(\sin\te\right)^{N-2}\,d\te} 
{\int_{0}^{\pi} e ^{-a R (1-\cos\te)} \left(\sin\te\right)^{N-2}\,d\te },
\end{equation*}
and hence
\begin{multline*}
0\leq \ve \log u^\ve_-(x) + \ve a d_\Ga(x) \leq \\
 \ve \log \left\{ \int_{0}^{\pi} \left(\sin\te\right)^{N-2}d\te\right\} - \ve \log \left\{ \int_{0}^{\pi} e ^{-a R (1-\cos\te)} \left(\sin\te\right)^{N-2} d\te \right\}.
\end{multline*}
We thus infer the desired asymptotics by just observing that
\begin{multline*}
\ve \log \left\{ \int_{0}^{\pi} e ^{-a R (1-\cos\te)} \left(\sin\te\right)^{N-2} d\te \right\} =\\
 \ve \log \left\{2^\frac{N-3}{2}\Ga\left(\frac{N-1}{2}\right) \left(\frac{\sqrt{\la}\,\ve}{R}\right)^\frac{N-1}{2} \left[ 1 + O(\ve) \right] \right\} \ \mbox {as } \ve\to 0^+,
\end{multline*}
thanks to \eqref{eq:asymptotics-to-infinity-I} of Lemma \ref{lem:bessel}.
\end{proof}



The next lemma provides the radial solutions of \eqref{elliptic-Pucci} for the exterior problem.
\begin{lem}[Solutions in the exterior of a ball]
\label{lem:exterior solution}
Let $\Om = \RR^N \setminus \ol{B_R}(0)$. Then, the solutions of \eqref{elliptic-Pucci} are the functions $u^\ve_\pm$ defined by
\begin{multline*}
u^\ve_-(x) =
\frac{\int_{0}^{\infty} e^{-\frac{|x|}{\ve} \frac{\cosh\te}{\sqrt{\la}}} \left(\sinh\te\right)^{-1+(N-1) \La/\la} d\te}
{\int_{0}^{\infty} e^{-\frac{R}{\ve} \frac{\cosh\te}{\sqrt{\la}}} \left(\sinh\te\right)^{-1+(N-1) \La/\la} d\te},
\\
u^\ve_+(x) =
\frac{\int_{0}^{\infty} e^{-\frac{|x|}{\ve} \frac{\cosh\te}{\sqrt{\La}}} \left(\sinh\te\right)^{-1+(N-1) \la/\La} d\te}
{\int_{0}^{\infty} e^{-\frac{R}{\ve} \frac{\cosh\te}{\sqrt{\La}}} \left(\sinh\te\right)^{-1+(N-1) \la/\La} d\te}, 
\end{multline*}
for any $x\in \ol\Om$ .
Moreover, we have that
\begin{equation*}
\ve\log  u^\ve_- + \frac{d_\Ga}{\sqrt{\la}} = O\left(\ve\right), \quad 
\ve\log  u^\ve_+ + \frac{d_\Ga}{\sqrt{\La}} = O\left(\ve\right),
\end{equation*}
uniformly on compact subsets of $\ol\Om$ as $\ve \to 0^+$.
\end{lem}

\begin{proof}
As before, we let $u^\ve_-(x)=u(r)$ with $r=|x|$. We then observe that $u(r)= f(r)/f(R)$, where this time $f$ is given in Lemma \ref{lem:bessel} with $a=1/(\sqrt{\la}\,\ve)$ and $b=-1+ (N-1)\,\La/\la$. Also, Lemma \ref{lem:bessel} informs us that $u_r(r) < 0$ and $u_{rr}(r) > 0$, and hence we have that
$$
\cM^-\left(\na^2 u^\ve_-(x)\right) = \la\,\left\{u_{rr}(r) + \frac{\La}{\la}\,\frac{N-1}{r}\, u_r(r) \right\}=\ve^{-2} u(r)=\ve^{-2} u^\ve_-(x).
$$
The boundary values of $u^\ve_-$ can be verified by inspection as before. 
\par
Now, the asymptotic formula follows since the function $u^\ve(x) \,e^{a(|x|-R)}$ is monotonic in $|x|$ and hence, for a given $\de>0$, we have that
\begin{equation*}
\frac{\int_{0}^{\infty} e^{a\de (1-\cosh\te)} \left(\sinh\te\right)^b d\te}
{\int_{0}^{\infty} e^{a R (1-\cosh\te)} \left(\sinh\te\right)^b d\te} \le 
u^\ve(x)\,e^{a(|x|-R)} \le 1 \ \mbox{ for } \ |x|\le\de.
\end{equation*}
Indeed, the last inequality gives that
\begin{equation*}
\ve \log \left\{\frac{\int_{0}^{\infty} e^{a\de (1-\cosh\te)} \left(\sinh\te\right)^b d\te}
{\int_{0}^{\infty} e^{a R (1-\cosh\te)} \left(\sinh\te\right)^b d\te}\right\} \le \ve \log\{u^\ve(x)\} + a \ve d_\Ga(x) \le 0.
\end{equation*}
We conclude by noticing that, by \eqref{eq:asymptotics-to-infinity} of Lemma \ref{lem:bessel}, it holds that
\begin{equation*}
\ve \log \left\{\frac{\int_{0}^{\infty} e^{a\de (1-\cosh\te)} \left(\sinh\te\right)^b d\te}
{\int_{0}^{\infty} e^{a R (1-\cosh\te)} \left(\sinh\te\right)^b d\te}\right\}  
= \ve \log \left\{ \left(\frac{R}{\de}\right)^\frac{(N-1)\La}{2\la} \left[1 + O(\ve)\right] \right\},
\end{equation*} 
as $\ve\to 0$.
\end{proof}

%

\section{Small diffusion asymptotics in general domains}
\label{sec:barriers}

In this section, we will use the radially symmetric solutions to construct useful barriers for problem \eqref{elliptic-Pucci} in quite general domains.

\subsection{Barriers from above and below}
We first recall a comparison principle for the differential equation in \eqref{elliptic-Pucci}, that works even in
unbounded domains. The result is an application of \cite[Theorem 2.2]{Sat} (see also \cite[Proposition 5.5]{Ko}).
\begin{lem}
\label{lem:elliptic comparison}
Let $\underline{u}$ and $\overline{u}$ be a sub-solution and a super-solution of the first equation in \eqref{elliptic-Pucci}, in the viscosity sense. 
Assume that $\underline{u}$ and $\overline{u}$ are continuous and bounded on $\ol{\Om}$, and that $\underline{u}\leq \overline{u}$ on $\Ga$.
Then, it holds that $\underline{u}\leq \overline{u}$ on $\ol\Om$.
\end{lem}

The following two lemmas are based on Lemmas \ref{lem:solution ball} and \ref{lem:exterior solution}.

\begin{lem}[Barriers from above]
\label{lem:barrier from above}
Assume that $u^\ve_\pm$ are the solutions of \eqref{elliptic-Pucci}. 
Then, it holds that
\begin{equation}
\label{eq:barrier from above}
\begin{aligned}
&&\ve \log u^\ve_-(x) +\frac{d_\Ga(x)}{\sqrt{\la}} \leq  \ve \log\left[ \frac{\int_{0}^{\pi} \left(\sin\te\right)^{N-2} d\te}{\int_{0}^{\pi} e^{-\frac{d_\Ga(x)}{\ve}\frac{1-\cos\te}{\sqrt{\la}}}\left(\sin\te\right)^{N-2} d\te } \right], \\
&&\ve \log u^\ve_+(x) +\frac{d_\Ga(x)}{\sqrt{\La}} \leq  \ve \log\left[ \frac{\int_{0}^{\pi} \left(\sin\te\right)^{N-2} d\te}{\int_{0}^{\pi} e^{-\frac{d_\Ga(x)}{\ve}\frac{1-\cos\te}{\sqrt{\La}}}\left(\sin\te\right)^{N-2} d\te } \right],
\end{aligned}
\end{equation}
for any $x\in\ol\Om$.
\end{lem}

\begin{proof}
Since for any $x\in\Om$, $B_R(x)$ with $R=d_\Ga(x)$ is contained in $\Om$ then, if $\overline{u}$ is the solution of \eqref{elliptic-Pucci} for $B_R(x)$,  an application of the comparison principle gives that $u^\ve \leq \overline{u}$ on $\ol{B_R(x)}$, and hence at $x$.
Thus, \eqref{eq:barrier from above} follows at once from Lemma \ref{lem:solution ball}.
\end{proof}

\begin{lem}[Barriers from below]
\label{lem:barrier from below}
Assume that $u^\ve_\pm$ are 
the solutions of \eqref{elliptic-Pucci} and
take a point $z\in\RR^N \setminus\ol\Om$. Then, it holds that
\begin{equation}
\begin{aligned}
\label{eq:below elliptic}
&\ve \log u^\ve_-(x) + \frac{|x-z| - d_\Ga(z)}{\sqrt{\la}}\ge \\
&\qquad\ve \log \left[ \frac{\int_0^\infty e^{-\frac{|x-z|}{\ve}\,\frac{\cosh\te -1 }{\sqrt{\la}}}\left(\sinh\te\right)^{-1+(N-1)\La/\la} d\te}{\int_0^\infty e^{-\frac{d_\Ga(z)}{\ve}\,\frac{\cosh\te -1 }{\sqrt{\la}}}\left(\sinh\te\right)^{-1+(N-1)\La/\la} d\te} \right], \\
&\ve \log u^\ve_+(x) + \frac{|x-z| - d_\Ga(z)}{\sqrt{\La}}\ge \\
&\qquad\ve \log \left[ \frac{\int_0^\infty e^{-\frac{|x-z|}{\ve}\,\frac{\cosh\te -1 }{\sqrt{\La}}}\left(\sinh\te\right)^{-1+(N-1)\la/\La} d\te}{\int_0^\infty e^{-\frac{d_\Ga(z)}{\ve}\,\frac{\cosh\te -1 }{\sqrt{\La}}}\left(\sinh\te\right)^{-1+(N-1)\la/\La} d\te} \right],
\end{aligned}
\end{equation}
for any $x\in\ol\Om$.
\end{lem}

\begin{proof}
If we set $R=d_\Ga(z)$, we have that $\Om $ is contained in $\RR^N\setminus \ol{B_R(z)}$.
We now apply the comparison principle to $u^\ve_-$ and the solution $\underline{u}$ of \eqref{elliptic-Pucci} relative to the domain $\RR^N\setminus \ol{B_R(z)}$. We thus obtain that $u^\ve \ge \underline{u}$ on $\ol\Om$
and the first statement in \eqref{eq:below elliptic} clearly follows from Lemma \ref{lem:exterior solution}.
\end{proof}

\subsection{Asymptotics in general domains}
\label{ssec:elliptic varadhan formula}

The desired asymptotic formulas of Theorem \ref{th:elliptic-Pucci} are obtained from Lemmas \ref{lem:barrier from above} and \ref{lem:barrier from below}.


\begin{proof}[Proof of Theorem \ref{th:elliptic-Pucci}, part (i)]
If $x\in\Ga$, the formula follows from the fact that $\ve\log u^\ve(x)$ is constantly equal to zero.
For any $x\in \Om$ and $z\in \RR^N\setminus\ol\Om$, by using both estimates in \eqref{eq:barrier from above} and \eqref{eq:below elliptic}, we have that
\begin{multline}
\label{eq:double estimate}
\ve \log\left[ \frac{\int_{0}^{\pi} \left(\sin\te\right)^{N-2} d\te}{\int_{0}^{\pi} e^{-\frac{d_\Ga(x)}{\ve}\frac{1-\cos\te}{\sqrt{\la}}}\left(\sin\te\right)^{N-2} d\te } \right] \ge  \ve \log u^\ve_-(x) +\frac{d_\Ga(x)}{\sqrt{\la}} \ge  \\
 \frac{d_\Ga(x)-|x-z| + d_\Ga(z)}{\sqrt{\la}} +\ve \log \left[ \frac{\int_0^\infty e^{-\frac{|x-z|}{\ve}\,\frac{\cosh\te -1 }{\sqrt{\la}}}\left(\sinh\te\right)^{-1+(N-1)\La/\la} d\te}{\int_0^\infty e^{-\frac{d_\Ga(z)}{\ve}\,\frac{\cosh\te -1 }{\sqrt{\la}}}\left(\sinh\te\right)^{-1+(N-1)\La/\la} d\te} \right].
\end{multline}
By letting $\ve \to 0^+$, we then obtain that 
\begin{multline*}
 \frac1{\sqrt{\la}} \{d_\Ga(x)-|x-z| + d_\Ga(z)\} \le \\ 
\liminf_{\ve\to 0^+} \left\{ \ve \log u^\ve_-(x) +\frac{d_\Ga(x)}{\sqrt{\la}} \right\} \le 
 \limsup_{\ve\to 0^+}  \left\{ \ve \log u^\ve_-(x) +\frac{d_\Ga(x)}{\sqrt{\la}} \right\} \leq 0.
\end{multline*}
Now, we take $y\in\Ga$ such that $|y-x|=d_\Ga(x)$ and we let $z\to y$. We obtain that both $d_\Ga(z)$ and $d_\Ga(x) - |x-z|$ vanish, and hence we conclude the proof.
\par
For the case of $u^\ve_+$, we proceed similarly.
\end{proof}

In order to obtain uniform estimates, we detail the definition of domain of class $C^{0,\om}$ outlined in the introduction.
Let $\om:(0,\infty)\to(0,\infty)$ be a strictly increasing continuous function such that $\om(0^+)=0$. We say that a domain $\Om$ is of class $C^{0,\om}$, if there exists a number $r>0$ such that, at every point $x_0\in\Ga$, there is a coordinate system $(y',y_N)\in\RR^{N-1}\times\RR$, and a function $\zi:\RR^{N-1}\to\RR$ such that
\begin{enumerate}[(i)]
\item
$B_r(x_0)\cap\Om=\{(y',y_N)\in B_r(x_0):y_N<\zi(y')\}$, 

\item
$B_r(x_0)\cap\Ga=\{(y',y_N)\in B_r(x_0):y_N=\zi(y')\}$, 

\item
$|\zi(y')-\zi(z')|\le\om(|y'-z'|)$ for all $(y',\zi(y')), (z',\zi(z'))\in B_r(x_0)\cap\Ga$.
\end{enumerate}
We then let $\psi_\om:[0,\infty)\to [0,\infty)$ be the function defined by
\begin{equation}
\label{psi function}
\psi_\om(\si) = \inf_{ s\geq 0} \sqrt{s^2 + [\om(s) -\si]^2} \ \mbox{ for } \ \si\geq 0.
\end{equation}


\begin{proof}[Proof of Theorem \ref{th:elliptic-Pucci}, part (ii)]
We fix a compact subset $K$ of $\ol\Om$ and we set 
$$
d=\max_K d_\Ga.
$$ 
\par
To obtain the uniform convergence in \eqref{eq:double estimate} we will choose $z=z_\ve$ independently on $x\in K$, as follows.
Since $\Om$ is of class $C^{0,\om}$, for a fixed $x\in K$, we take $y\in\Ga$ minimizing the distance to $x$, and 
consider a coordinate system in $\RR^{N-1}\times\RR$ such that $y=(0', 0)$. If we take $z_\ve=(0',\ve)$, then $z_\ve\in\RR^N\setminus\ol\Om$ when $\ve\le d$ is sufficiently small. 
Also, we have that $|x-z_\ve|\leq d_\Ga(x)+\ve \leq 2 d$ and that $d_\Ga(z_\ve) \geq \psi_\om(\ve)$.
\par
Hence, \eqref{eq:double estimate} reads as
\begin{multline*}
 -\frac{\ve}{\sqrt{\la}}+ \ve \log \left[ \frac{\int_{0}^{\infty} e^{-2d\,\frac{\cosh\te -1 }{\ve\sqrt{\la}}}\left(\sinh\te\right)^{-1+(N-1)\,\La/\la} d\te}{\int_{0}^{\infty} e^{-\frac{\cosh\te -1 }{\ve\sqrt{\la}}\psi_\om(\ve)}\left(\sinh\te\right)^{-1+(N-1)\,\La/\la} d\te} \right] \le\\
\ve \log u^\ve_-(x) +\frac{d_\Ga(x)}{\sqrt{\la}} \leq  \ve \log\left[ \frac{\int_{0}^{\pi} \left(\sin\te\right)^{N-2}\,d\te}{\int_{0}^{\pi} e^{-\frac{d}{\ve\sqrt{\la}}(1-\cos\te)}\left(\sin\te\right)^{N-2}\,d\te } \right].
\end{multline*}
A similar chain of inequalities can be obtained for $u^\ve_+$, by simply switching the roles of $\la$ and $\La$. In any case, from \eqref{eq:asymptotics-to-infinity-I} in Lemma \ref{lem:bessel}, we easily infer that the last term is $O\left(\ve\log \ve\right)$ as $\ve \to 0^+$. 
\par 
Next, to take care of the first term, we consider the quantity 
$$
\eta_\ve(\la,\La)=
\frac{\int_{0}^{\infty} e^{-2d\,\frac{\cosh\te -1 }{\ve\sqrt{\la}}}\left(\sinh\te\right)^{-1+(N-1)\,\La/\la} d\te}{\int_{0}^{\infty} e^{-\frac{\cosh\te -1 }{\ve\sqrt{\la}}\psi_\om(\ve)}\left(\sinh\te\right)^{-1+(N-1)\,\La/\la} d\te},
$$
for the case of $u^\ve_-$, and apply \eqref{eq:asymptotics-to-infinity} with $\si=2d/(\ve\sqrt{\la})$ at the numerator and \eqref{eq:asymptotics-to-zero} with $\si=\psi_\om(\ve)/(\ve\sqrt{\la})$ at the denominator. In both cases we set $a=1$ and $b=-1+(N-1)\,\La/\la=N-2+(N-1)(\La/\la-1)\ge 0$.
Notice that $b =0$ if and only if $N=2$ and $\la=\La$. Hence, if $N\neq 2$ or $\la \neq \La$, from  \eqref{eq:asymptotics-to-zero} with $b>0$ we see that 
$$
\eta_\ve(\la,\La)=\frac{\Ga\left(\frac{b+1}{2}\right)}{2\Ga\left(b\right) d^\frac{b+1}{2} \la^\frac{b-1}{4}}\,\ve^\frac{1-b}{2}\,\psi_\om(\ve)^b \left\{1+o(1)\right\}
 \ \mbox{ as } \ \ve\to 0^+,
$$
which gives that
$$
\ve\log\eta_\ve(\la,\La)=O\left(\ve \log \psi_\om(\ve)\right) \ \mbox{ as } \ \ve\to 0^+.
$$
If $N=2$ and $\la=\La$, we apply \eqref{eq:asymptotics-to-zero} with $b=0$ and hence we infer that
$$
\eta_\ve(\la,\La)= \frac{2^{-\frac{1}{2}} \Ga\left(\frac{1}{2}\right) \left( \frac{2d}{\ve\sqrt{\la}}\right)^{-\frac{1}{2}}\left\{1+o(1)\right\}}{\log\left(\frac{\ve\sqrt{\la}}{\psi_\om(\ve)}\right) +  O(1)} \ \mbox{ as } \ \ve\to 0^+,
 $$
which gives that
$\ve\log\eta_\ve(\la,\La)= O\left(\ve\log |\log \psi_\om(\ve)|\right)$
 as $\ve\to 0^+$.
\par
In the case of $u^\ve_+$, we must work with $\eta_\ve(\La,\la)$, instead. In particular, we must choose $b=-1+ (N-1)\,\la/\La$, and this means that $b$ may be positive, zero, and also negative. The cases $b\ge 0$, that correspond to the range $\La\le (N-1)\,\la$, can be settled by arguing as above. 
\par
It only remains to settle the case in which $-1<b<0$, that occurs when $\La > (N-1)\,\la$. We can still apply Lemma \ref{lem:bessel} and obtain that
$$
\eta_\ve(\La,\la)= \frac{\sin\left(\frac{b\pi}{2}\right)\Ga(1+b/2)\la^\frac{b+1}{4}}{\sqrt{\pi}\,d^\frac{b+1}{2}}\,\ve^\frac{b+1}{2}\,\left\{1+o(1)\right\}
\ \mbox{ as } \ \ve\to 0^+,
$$
which gives that 
$\ve\log \eta_\ve(\La,\la)=O\left(\ve \log\ve\right)$ as $\ve\to 0^+$.
\end{proof}

\begin{rem}
{\rm
Observe that, if we set $N^+_{\la,\La}=1+(N-1)\,\la/\La$, part (ii) of Theorem \ref{th:elliptic-Pucci} can be rephrased as follows:
\begin{equation*}
\ve \log u^\ve + \frac{d_\Ga}{\sqrt{\La}} =
\begin{cases}
O\left(\ve \log \ve\right) \ &\mbox{ if } \ N^+_{\la,\La}<2,\\
O\left(\ve \log |\log\psi_\om(\ve)|\right) \ &\mbox{ if } \ N^+_{\la,\La}=2,\\
O\left( \ve \log\psi_\om(\ve)\right) \ &\mbox{ if } \ N^+_{\la,\La}>2.\\
\end{cases} 
\end{equation*}
It is worth noting that $N^+_{\la,\La}$ comes into play to determine the threshold for the existence of non-trivial solutions of the problem 
$$
\cM^+\left(\na^2 u\right) + u^p = 0 \ \mbox{ in } \ \RR^N, \quad u\geq 0 \ \mbox{ in } \ \RR^N,
$$
as shown in \cite{Cutri2000, Felmer2003}.
}
\end{rem}


%


\section{Small time asymptotics in general domains} 

In this section, we consider the solution $v=v(x,t)$ of  the initial-boundary value problem \eqref{parabolic-Pucci}:
$$
v_t- \cM^\pm(\na^2 v)=0 \ \mbox{ in } \ \Om\times (0,\infty),  \ v=0 \ \mbox{ on } \ \Om\times \{ 0\},  \ v=1 \ \mbox{ on } \ \Ga\times (0,\infty).
$$

The following parabolic comparison principle is a corollary of \cite[Theorem 2.1]{GGIS} (see also \cite[Theorem A.1]{BM-JMPA}). We stress the fact that the relevant functions need not be totally bounded.

\begin{lem}
\label{lem:parabolic comparison}
Let $u(x,t)$ and $w(x,t)$ be a sub-solution and a super-solution in the viscosity sense of the first equation in \eqref{parabolic-Pucci}. Assume that $u$ and $w$ are continuous. Moreover, suppose that $u$ and $-w$ are bounded from above.
\par
Suppose that $u\leq w$ on $\left(\Ga\times (0,\infty)\right)\cup \left(\Om \times \{0\}\right)$. Then, $u\leq w$ on $\ol\Om\times (0,\infty)$.
\end{lem}

\subsection{Barriers from above and below}
\label{ssec:barriers}

The next lemma provides barriers from above for $v^\pm$, based essentially on estimating $v^\pm$ in balls by means of $u^\ve_\pm$. In \cite{BM-JMPA} the analog in the case of the game-theoretic $p$-laplacian was obtained by employing a version of the {\it Laplace transform} (see \cite[Lemma 2.8 and formula (2.17)]{BM-JMPA}). We need here a different proof which relies on an application of Lemma \ref{lem:parabolic comparison}.

\begin{lem}[Parabolic barriers from above]
\label{lem:parabolic from above}
Assume that $v^{\pm}$ satisfies \eqref{parabolic-Pucci}.
\par
For any $(x,t) \in \ol\Om\times (0,\infty)$ it holds that
\begin{equation}
\label{eq:parabolic barrier from above}
\begin{aligned}
&4 t \log v^-(x,t) \leq -\frac{d_\Ga(x)^2}{\la} + 4 t \log \left[ \frac{\int_{0}^{\pi} \left(\sin\te\right)^{N-2} d\te}{\int_{0}^{\pi} e^{-\frac{d_\Ga(x)^2}{2\la t}(1-\cos\te)}\left(\sin\te\right)^{N-2} d\te} \right],\\
&4 t \log v^+(x,t) \leq -\frac{d_\Ga(x)^2}{\La} + 4 t \log \left[ \frac{\int_{0}^{\pi} \left(\sin\te\right)^{N-2} d\te}{\int_{0}^{\pi} e^{-\frac{d_\Ga(x)^2}{2\La t}(1-\cos\te)}\left(\sin\te\right)^{N-2} d\te} \right].
\end{aligned}
\end{equation}
\end{lem}

\begin{proof}
We prove \eqref{eq:parabolic barrier from above} for the case $v^-$. With obvious adjustments, the proof for $v^+$ runs similarly. First, we observe that if $x\in\Ga$ then \eqref{eq:parabolic barrier from above} is trivially satisfied. If $x\in\Om$ we argue as follows. Consider the ball $B\subset \Om$ centered at $x$ with radius $d_\Ga(x)$. Fix the parameter $\ve>0$. Define, for $x'\in \ol{B}$ and $t>0$,
$$
w(x',t)=e^{t/\ve^2} u_-^\ve(x'),
$$
where $u_-^\ve$ is the solution of \eqref{elliptic-Pucci} in $B$. Since $\cM^-$ is positively one-homogeneous and $u_-^\ve$ solves \eqref{elliptic-Pucci}, it is a plain inspection to verify that
\begin{eqnarray*}
&w_t - \cM^-\left(\na^2 w\right)=0 \ &\mbox{ in } \ B\times(0,\infty),\\
&w = e^\frac{t}{\ve^2} > 1 \ &\mbox{ on } \ \pa B\times(0,\infty),\\
&w=u^\ve >0 \ &\mbox{ on } \ B\times\{0\}.
\end{eqnarray*}
\par
Lemma \ref{lem:parabolic comparison} gives that $v^-\leq w$ in $\ol{B}\times(0,\infty)$ and in particular at the center of $B$, that is $v^-(x,t)\leq w(x,t)$, for any $t>0$. By recalling \eqref{radial solution}, we obtain that, for any $t,\ve>0$, 
$$
4t \log v^-(x,t)\leq 4t \left(\frac{t}{\ve^2}-\frac{d_\Ga(x)}{\ve\sqrt{\la}}\right) + 4 t\log \left[ \frac{\int_{0}^{\pi} \left(\sin\te\right)^{N-2} d\te}{\int_{0}^{\pi} e^{-\frac{d_\Ga(x)}{\ve \sqrt{\la}}(1-\cos\te)}\left(\sin\te\right)^{N-2} d\te} \right].
$$
Hence, after choosing $\ve = \frac{2\sqrt{\la}}{d_\Ga(x)} t $, we get \eqref{eq:parabolic barrier from above} for $v^-$. 
\end{proof}

In the next lemma we present global sub-solutions for the differential equation in \eqref{parabolic-Pucci} that are instrumental to construct barriers from below for $v^\pm$.

\begin{lem}
\label{fundamental solution}
The functions $\Phi^\pm$, defined for any $x\in\RR^N$ and $t>0$ by
\begin{equation}
\label{global solutions}
\Phi^-(x,t) = t^{-\frac{N\La}{2\la}} e^{-\frac{|x|^2}{4\la t}} \ \mbox{ and } \ \Phi^+(x,t) = t^{-\frac1{2}-\frac{(N-1)\la}{2\La}}  e^{-\frac{|x|^2}{4\La t}},
\end{equation}
satisfy
\begin{equation}
\label{global equations}
\begin{aligned}
&&\Phi^-_t - \cM^-\left(\na^2 \Phi^-\right) \leq  0 \ \mbox{ in } \ \RR^N\times(0,\infty),\\
&&\Phi^+_t - \cM^+\left(\na^2 \Phi^+\right) \leq 0 \ \mbox{ in } \ \RR^N\times(0,\infty).
\end{aligned}
\end{equation}
\end{lem}

\begin{proof}
Inequalities \eqref{global equations} follow from direct computations by taking into consideration \eqref{radial pucci minus}-\eqref{radial pucci plus}. Indeed, for $\Phi^-$ we get that
\begin{equation*}
\Phi_t^-(x,t) - \cM^-\left(\na^2 \Phi^-(x,t)\right)=
\begin{cases}
\displaystyle
\frac{\la-\La}{2\la t}\Phi^-(x,t) \ &\mbox{ if } \ |x|^2 \geq 2\la t\\[10pt]
\displaystyle
\frac{|x|^2(\la-\La)}{4\la^2 t}\Phi^-(x,t)\ &\mbox{ elsewhere.}
\end{cases}
\end{equation*}
Analogously, for $\Phi^+$ we get that
\begin{equation*}
\Phi_t^+(x,t) - \cM^+\left(\na^2 \Phi^+(x,t)\right)=
\begin{cases}
\displaystyle
0\ &\mbox{ if } \ |x|^2 \geq 2\La t\\[10pt]
\displaystyle
\frac{\La - \la}{4\La^2 t^2}\left(|x|^2 - 2\La t\right)\Phi^+(x,t) \ &\mbox{ elsewhere.}
\end{cases}
\end{equation*}
The claim then follows by an inspection.
\end{proof}


\begin{lem}[Parabolic barriers from below]
\label{lem:parabolic from below}
Assume that $v^\pm$ is the solution of \eqref{parabolic-Pucci}. Let $z\in \RR^N\setminus \ol\Om$ and $\de = d_\Ga(z)$.
\par
Then, we have that
\begin{equation}
\label{eq:parabolic barrier from below}
\begin{aligned}
& v^-(x,t) \geq A_-  t^{-\frac{N\La}{2\la}}  e^{-\frac{|x-z|^2}{4\la t}} \ \mbox{ in } \ \ol\Om\times \left(0,\infty\right),\\
& v^+(x,t) \geq A_+ t^{-\frac1{2}-\frac{(N-1)\la}{2\La}}  e^{-\frac{|x-z|^2}{4\La t}} \ \mbox{ in } \ \ol\Om\times(0,\infty),
\end{aligned}
\end{equation}
where
\begin{equation}
A_-=
\left\{
\frac{\de^2 e}{2N\La}\right\}^{\frac{N\La}{2\la}}, \qquad  A_+=
\left\{
\frac{\de^2 e}{2\left[(N-1) \la+\La\right]}\right\}^{\frac{(N-1)\la + \La}{2\La}}.
\end{equation}
\end{lem}
\begin{proof}
We prove the first formula in \eqref{eq:parabolic barrier from below}. The other formula follows similarly. We apply Lemma \ref{lem:parabolic comparison} to $v^-$ and the function $\Om\times(0,\infty)\ni(x,t)\mapsto A_-\Phi^-(x-z,t)$, obtained from $\Phi^-$ in \eqref{global solutions} after translating by $z$ and multiplying by the positive real number $A_-$ defined by
$$
\frac{1}{A_-}=\max\{\Phi^-(x,t): (x,t)\in\ol\Om\times (0,\infty)\}=\left\{ \frac{\de^2 e}{2 N\La}\right\}^{-\frac{N\La}{2\la}}.
$$
\par
Since $\cM^-$ is translation invariant and positively one-homogeneous, from \eqref{global equations} it follows that $A_-\Phi^-(x-z,t)$ is a sub-solution of the first equation in \eqref{parabolic-Pucci}.
A direct inspection shows that $A_-\Phi^-(x-z,t)$ tends to $0$ as $ t\to 0^+$, for any $x\in \Om$. Also, we have that $A_- \Phi^-(x-z,t) \leq 1$ on $\Ga\times (0,\infty)$ by construction. Then Lemma \ref{lem:parabolic comparison} gives \eqref{eq:parabolic barrier from below}.
\end{proof}

\subsection{Asymptotics in general domains}
\label{ssec:parabolic varadhan formula}


In this subsection, we prove Theorem \ref{th:parabolic-Pucci}. 

\begin{proof}[Proof of Theorem \ref{th:parabolic-Pucci} part (i)]
We prove \eqref{parabolic pucci asymp} only for $v^-$, since the proof for $v^+$ runs similarly with obvious adjustments. For any $x\in\ol\Om$ and $z\in \RR^N\setminus\ol\Om$, from Lemmas \ref{lem:parabolic from above} and  \ref{lem:parabolic from below}, it follows that 
\begin{multline}
\label{eq:ppucci double estimate}
4t\log A_- -2\frac{N\La}{\la} t \log t - \frac{|x-z|^2}{\la} \leq 4 t \log v^-(x,t) \leq \\
 -\frac{d_\Ga(x)^2}{\la} + 4 t \log \left[ \frac{\int_{0}^{\pi} \left(\sin\te\right)^{N-2} d\te}{\int_{0}^{\pi} e^{-\frac{d_\Ga(x)^2}{2\la t}(1-\cos\te)}\left(\sin\te\right)^{N-2} d\te} \right],
\end{multline}
for any $t>0$. Thus, as $t\to 0^+$ we get:
\begin{equation*}
-\frac{|x-z|^2}{\la} \leq 4\liminf_{t \to 0^+} \{ t \log v^-(x,t)\} \leq 4 \limsup_{t\to 0^+} \{t\log v^-(x,t)\}\leq -\frac{d_\Ga(x)^2}{\la}.
\end{equation*}
Formula \eqref{parabolic pucci asymp} then follows by letting $z\to y\in\Ga$, where $y\in\Ga$ is such that $|y-x|=d_\Ga(x)$.
\end{proof}

Recall that, given $\Om\in C^{0,\om}$, the function $\psi_\om$ is defined by \eqref{psi function}.

\begin{proof}[Proof of Theorem \ref{th:parabolic-Pucci} part (ii)]
We prove \eqref{uniform ppucci asymp} only for $v^-$, since for $v^+$ we argue similarly. Let $K$ be a compact subset of $\ol\Om$ and let $d$ be the number already defined in the proof of Theorem \ref{th:elliptic-Pucci}.
\par 
For a given $x\in K$, let $y\in\Ga$ be a point minimizing the distance from $x$ to $\Ga$ and consider the coordinate system in $\RR^{N-1}\times\RR$ such that $y=(0',0)$. In this coordinate system, for $t>0$, we choose  $z_t=(0',t)$. Then, $z_t \in\RR^N\setminus\ol\Om$ when $t\leq d$ is sufficiently small. This implies that $|y-z_t|=t$, $d_\Ga(z_t) \geq \psi_\om(t)$ and that $d_\Ga(x)^2 - |x-z_t|^2 \geq d_\Ga(x)^2 -\left(d_\Ga(x) +t\right)^2 \geq - \left(t+2d\right)t$. 
 \par 
 Hence, from \eqref{eq:ppucci double estimate} it follows that
\begin{multline}
\label{uniform ineq}
2 \frac{N\La}{\la}t \log \left\{ \frac{ \psi_\om(t)^2 e}{2N\La} \right\} -2\frac{N\La}{\la} t \log t - \frac{\left(t + 2 d\right) t}{\la}\leq \\
 4 t \log v^-(x,t) + \frac{d_\Ga(x)^2}{\la} \leq 4 t \log \left[ \frac{\int_{0}^{\pi} \left(\sin\te\right)^{N-2} d\te}{\int_{0}^{\pi} e^{-\frac{d^2}{2\la t}(1-\cos\te)}\left(\sin\te\right)^{N-2} d\te} \right].
\end{multline}
\par
A direct inspection reveals that the left-hand side of \eqref{uniform ineq} vanishes uniformly, as $t\to 0^+$, as $O\left(t \log \psi_\om(t)\right)$. The application of \eqref{eq:asymptotics-to-infinity-I} in Lemma \ref{lem:bessel} with $\si=d^2/(2\La t)$ gives that the right-hand side of \eqref{uniform ineq} vanishes as $O\left(t \log t\right)$, for $t\to 0^+$. Thus, we have obtained \eqref{uniform ppucci asymp}.
\end{proof}

\section{Formulas for $q$-means}
\label{sec:q-mean-formulas}
Before moving on, we recall the definition of $q$-means on balls. Set $1\leq q\leq\infty$. Consider a continuous function $u: \ol{B}\to\RR$ over a ball $B$. The $q$-mean of $u$ on $B$ is the unique minimum point of the function $\RR\ni \mu \mapsto
\Vert u - \mu \Vert_{L^q\left(B\right)}$.
\par
In Theorems \ref{th:qmeans elliptic} and \ref{th:qmeans parabolic} we provide the respective asymptotic profiles as $\ve\to 0^+$ and $t\to 0^+$ of the $q$-means of $u_\pm^\ve$ and $v^\pm\left(\cdot,t\right)$ on balls touching the boundary $\Ga$. These formulas extend those obtained by Berti and Magnanini in \cite[Theorem 3.5]{BM-AA} and \cite[Theorem 3.5]{BM-JMPA} concerning the game-theoretic $p$-laplacian. Since the proofs of Theorems \ref{th:qmeans elliptic} and \ref{th:qmeans parabolic} substantially mimic those in \cite{BM-AA} and \cite{BM-JMPA}, respectively, we only summarize here the main steps.

Throughout this section, we assume that $\Om$ is of class $C^2$ and that $x\in\Om$ is such that $\ol{B_R(x)}\cap\left(\RR^N\setminus\Om\right)=\{z_x\}$, for some $R>0$ and $z_x\in\Ga$ such that $\ka_j(z_x) < \frac1{R}$ for $j=1,\dots,N-1$. Also, we set
$$
\Pi_0(z_x)=\prod_{j=1}^{N-1}\left[1-Rk_j(z_x)\right].
$$


%

\begin{thm}
  \label{th:qmeans elliptic}
Let $u_\pm^\ve$ be the solution of \eqref{elliptic-Pucci}. For $1<q\leq\infty$, let $\mu^\pm_{q,\ve}(x)$ be the $q$-mean of $u_\pm^\ve$ on $B_R(x)$.
\par 
Then, for any $1<q<\infty$, it holds:
\begin{equation}
  \label{q-mean elliptic balls}
  \begin{aligned}
\lim_{\ve\to 0^+}
\left(\frac{R}{\ve}\right)^{\frac{N+1}{2(q-1)}}\mu^-_{q,\ve}(x)=
c_{N,q}\,\left\{ \frac{\Pi_0(z_x)}{\la^{\frac{N+1}{2}}}\right\}^{-\frac{1}{2(q-1)}},\\
  \lim_{\ve\to 0^+}
\left(\frac{R}{\ve}\right)^{\frac{N+1}{2(q-1)}}\mu^+_{q,\ve}(x)=
c_{N,q}\,\left\{ \frac{\Pi_0(z_x)}{\La^{\frac{N+1}{2}}}\right\}^{-\frac{1}{2(q-1)}},
\end{aligned}
\end{equation}
where
$$
c_{N,q}= \left[\frac{2^{-\frac{N+1}{2}}N!}{(q-1)^{\frac{N+1}{2}}\Ga\left(\frac{N+1}{2}\right)}\right]^\frac1{q-1}.
$$
(Here, $\Ga\left(\frac{N+1}{2}\right)$ is the {\it Euler's gamma function} evaluated at $\frac{N+1}{2}$.)
\par
In the case $q=\infty$, we have:
$$
\lim_{\ve \to 0^+} \mu^\pm_{\infty,\ve}(x) = \frac1{2}.
$$
\end{thm}

\begin{proof}
As usual, we only consider the case  with the superscript ``$-$''. As in \cite[Lemma 3.1]{BM-AA}, since $\Om$ is of class $C^2$ and $B_R(x)$ is bounded, we are able to improve the barriers given in Lemmas \ref{lem:solution ball} and \ref{lem:exterior solution}. The restriction of $u^\ve_-$ to $B_R(x)$ lies between two functions that spatially depend only on the distance to $\Ga$:
$$
f_\ve^1\left(\frac{d_\Ga}{\sqrt{\ve^2\la}}\right) \leq u^\ve_- \leq f_\ve^2\left(\frac{d_\Ga}{\sqrt{\ve^2\la}}\right)\ \mbox{ in } \ B_R(x).
$$
Here, for $\si\geq 0$, $f_\ve^1$ and $f_\ve^2$ (which correspond to $U^\ve$ and $V^\ve$ in the first part of \cite[Lemma 3.1]{BM-AA}) are defined by:
$$
f^1_\ve(\si)=\frac{\int_{0}^{\infty}e^{-\si\cosh\te}\,\left\{e^{-\frac{r_e\cosh\te}{\ve\sqrt{\la}}}\left(\sinh\te\right)^{-1+(N-1)\La/\la}\,d\te\right\}}{\int_{0}^{\infty}e^{-\frac{r_e\cosh\te}{\ve\sqrt{\la}}}\left(\sinh\te\right)^{-1+(N-1)\La/\la}\,d\te}
$$
and
$$
f^2_\ve(\si)=\frac{\int_{0}^{\pi}e^{-\si\cos\te}\,\left\{e^{\frac{r_i\cos\te}{\ve\sqrt{\la}}}\left(\sin\te\right)^{N-2}\,d\te\right\}}{\int_{0}^{\pi}e^{\frac{r_i\cos\te}{\ve\sqrt{\la}}}\left(\sin\te\right)^{N-2}\,d\te},
$$
where the positive real numbers $r_i$ and $r_e$ are the radii such that the projection of $\ol{B_R(x)}$ on $\Ga$ satisfies the interior and the exterior sphere condition, respectively.
\par 
With these barriers in mind, we just proceed as in \cite[Theorem 3.5]{BM-AA}. Since the $q$-means are monotonic with respect to the pointwise order between functions, formulas \eqref{q-mean elliptic balls} and the one in the case $q=\infty$ result after computing them for both the functions $f_\ve^1\left(d_\Ga/\sqrt{\ve^2\la}\right)$ and $f_\ve^2\left(d_\Ga/\sqrt{\ve^2\la}\right)$. In the case $1<q<\infty$, the desired asymptotics are a consequence of applications of the co-area formula and the geometrical lemma \cite[Lemma 2.1]{MS-IUMJ}, which generates the term $\Pi_0(z_x)$. The case $q=\infty$ plainly follows from the fact that the $\infty$-mean of a function is the arithmetic mean of its supremum and its infimum.
\end{proof}



\begin{thm}
\label{th:qmeans parabolic}
Suppose that $v^\pm$ is the solution of \eqref{parabolic-Pucci}. For $1<q \leq \infty$, let $\mu^\pm_q(x,t)$ be the $q$-mean of $v^\pm\left(\cdot, t\right)$ on $B_R(x)$.
\par 
Then, for $1<q<\infty$, it holds that
\begin{equation}
\label{eq:parabolic balls qmean}
\begin{aligned}
\lim_{t\to 0^+} \left(\frac{R^2}{t}\right)^\frac{N+1}{4(q-1)}\!\!\!\!\mu^-_q(x,t)= 
C_{N,q}\, \left\{ \frac{\Pi_0(z_x)}{\la^{\frac{N+1}{2}}}\right\}^{-\frac{1}{2(q-1)}}, \\
\lim_{t\to 0^+} \left(\frac{R^2}{t}\right)^\frac{N+1}{4(q-1)}\!\!\!\!\mu^+_q(x,t)=
C_{N,q}\, \left\{ \frac{\Pi_0(z_x)}{\La^{\frac{N+1}{2}}}\right\}^{-\frac{1}{2(q-1)}},
\end{aligned}
\end{equation}
where 
$$
C_{N,q} = \left[\frac{N!\,\int_0^\infty \Erfc(\si)^{q-1} \si^\frac{N-1}{2} d\si}{\,\Ga \left( \frac{N+1}{2} \right)^2}\right]^\frac1{q-1}.
$$
Here, $\Erfc$ is the {\it complementary error function}, defined by $$
\Erfc(\si)=\frac2{\sqrt{\pi}}\int_{\si}^{\infty}e^{-s^2}\,ds
$$
for $\si\in\RR$.
\par
In the case $q=\infty$, we have:
$$
\lim_{t\to 0^+} \mu_\infty^\pm(x,t) = \frac1{2}.
$$
\end{thm}

\begin{proof}
We follow the proof of \cite[Theorem 3.5]{BM-JMPA}. After some manipulations from \eqref{uniform ppucci asymp} as in \cite[Corollary 2.12 and Theorem 3.5]{BM-JMPA}, since $\Om$ is of class $C^2$ and $\ol{B_R(x)}\subset \ol\Om$ is compact we argue that, for any $(y,t)\in\ol{B_R(x)}\times (0,\infty)$,
\begin{equation*}
\label{parabolic barriers}
\Erfc\left(\frac{d_\Ga(y)}{2\sqrt{\la \,t}}+\eta(t)\right) \leq v^-(y,t) \leq \Erfc\left(\frac{d_\Ga(y)}{2\sqrt{\la\, t}}-\eta(t)\right),
\end{equation*}
where $\eta(t):(0,\infty)\to(0,\infty)$ is such that $\eta(t) = O\left(\sqrt{t}\log t\right)$, as $t\to 0^+$.
\par
As in the elliptic case (Theorem \ref{th:qmeans elliptic}), the statement follows from the monotonicity of the $q$-means with respect to the pointwise order between functions and the computation of the desired asymptotics for both barriers. In particular, the case $1<q<\infty$ descends from applications of the co-area formula and the geometrical lemma in \cite[Theorem 2.1]{MS-IUMJ} while the case $q=\infty$ is given directly by the definition of $\infty$-mean.
\par
Formulas for $\mu_q^+(x,t)$ are treated similarly.

\section*{Acknowledgments}
This paper was partially supported by the Gruppo Nazionale di Analisi Matematica, Probabilit\`a e Applicazioni (GNAMPA) of the Istituto Nazionale di Alta Matematica (INdAM). D.B. was supported by the Programa de Excelencia Severo Ochoa SEV-2015-0554 of the Ministerio de Economia y Competitividad.

\end{proof}



\bibliographystyle{abbrv2}
\bibliography{mybib}{}

 \end{document}